\documentclass[10pt,a4paper,reqno]{amsart}
\usepackage[latin1]{inputenc}
\usepackage[T1]{fontenc}
\usepackage{lmodern}
 \usepackage[english]{babel,minitoc}
\usepackage{hyperref}
\usepackage{amsfonts}
\usepackage{amsmath}

\numberwithin{equation}{section}\theoremstyle{plain}
\newtheorem{theorem}{Theorem}[section]

\newtheorem{lemma}[theorem]{Lemma}

\newtheorem{definition}[theorem]{Definition}
\newtheorem{remark}[theorem]{Remark}

\address{\begin{center}{\small Department of Mathematics and Computer Sciences, Faculty of Sciences-Mekn\`{e}s,\\
Equipe d'Analyse Harmonique et Probabilit\'{e}s, University Moulay Isma\"{\i}l,\\
BP 11201 Zitoune, Meknes, Morocco}
\end{center}}

\begin{document}

\title[Local uncertainty principles for the two-sided Gabor Quaternion Fourier transform ]
{Local uncertainty principles  for the two-sided Gabor Quaternion Fourier transform }

\author[  M. El kassimi and S. Fahlaoui ]{ M. El kassimi  \quad and \quad S. Fahlaoui }

\address{Sa\"{\i}d Fahlaoui}  \email{s.fahlaoui@fs.umi.ac.ma}

\address{Mohammed El kassimi} \email{m.elkassimi@edu.umi.ac.ma}

\thanks{Mohammed El kassimi: \texttt{m.elkassimi@edu.umi.ac.ma}}
\thanks{Sa\"{\i}d Fahlaoui: \texttt{s.fahlaoui@fs.umi.ac.ma}}

\maketitle
\begin{abstract}\leavevmode\par
Regarding the important applications of Gabor transform in time-frequency analysis and signal analysis, Actually, in this paper,  we consider the Gabor quaternion Fourier transform (GQFT), and we prove a version  of Benedicks-type  uncertainty principle for GQFT and some  local concentration uncertainty principles.

\end{abstract} {\it keywords:} Quaternion algebra, Quaternion Fourier transform, Gabor Fourier transform, concentration theorem,  Benedicks theorem.\\
\section{Introduction}
The two sided quaternion Fourier transform has become increasingly important tool in image processing. Thus, the two-sided quaternion Fourier transform gives us a simple representation of signals with several components that can be controlled simultaneously. The quaternion Fourier transform QFT was first introduced by Ell \cite{Ell}. In \cite{Hitzar}, Hitzer has proved  important properties  of QFT. In \cite{Sangwine}, the authors used the QFT in color image analysis. \\
In 1940 \cite{Gabor1}, Dennis Gabor has given a powerful tool in signal analysis. So, he introduced a windowed function for studying a time-frequency signals. This idea becomes a useful tool to obtain information about a signal in limited region. Motivated by the applications of QFT and the Gabor Fourier transform (GFT) cited above, in \cite{Elkassimi} the others have given an extension for GFT to the quaternion case. They have defined the two-sided Gabor quaternion Fourier transform (GQFT). Some useful results of GQFT are derived, like Plancherel and reconstruction formulas. Also, in \cite{Elkassimi} the authors have demonstrated a version of the Heisenberg uncertainty principle and Logarithmic inequality for the GQFT.\\
The uncertainty principle state that, we cannot give simultaneously the position and momentum  of a particle. Therefore, if try to limit the region of one we lose control of the other. The uncertainty principles have many applications in quantum physics and signal analysis. There are many versions of uncertainty principles, for example, Donoho-stark \cite{Donoho}, Benedicks theorem \cite{Benedicks}.
The aim of this paper is to demonstrate some uncertainty for the GQFT.
The paper is organized as follows. In the second section, we remind some harmonic analysis properties for the two-sided quaternion Gabor Fourier transform proved in \cite{Elkassimi}. In the third section, we give a version of the  Benedicks-type theorem and  some local uncertainty principles.

\subsection{Definition and properties of quaternion $\mathbb{H}$}: \\
Considering the the classical notations, the quaternion algebra $\mathbb{H}$ is the set of all elements $q$ such that
$$q=q_1+iq_2+jq_3+kq_4\in \mathbb{H},~~ q_1,q_2,q_3,q_4\in\mathbb{R},$$
 with i, j and k  are  three imaginary units obey the Hamilton's multiplication rules, \begin{equation}\label{equ1} ij=-ji=k, ~~ jk=-kj=i, ~~ki=-ik=j \end{equation}
$$i^2=j^2=k^2=ijk=-1$$
Due  to \ref{equ1} $\mathbb{H}$ is non-commutative algebra.

The conjugate of quaternion $q$ is obtained by changing the sign of the pure part, i.e.
$$\overline{q}=q_1-iq_2-jq_3-kq_4$$
The quaternion conjugation is a linear anti-involution $$\overline{\overline{p}}=p,~~~~ \overline{p+q}=\overline{p}+\overline{q},~~ \overline{pq}=\overline{q} \overline{p},~ \forall p,q\in \mathbb{H}$$
The modulus $|q|$ of a quaternion q is giving by $$|q|=\sqrt{q\overline{q}}=\sqrt{q_1^2+q_2^2+q_3^2+q_4^2},~~ |pq|=|p||q|. $$
the modulus of  $\mathbb{H}$  has the following properties,
$$|pq|=|p||q|, |q|=|\overline{q}|, p, q \in \mathbb{H}$$
In particular, when $q=q_1$ is a real number, the module $|q|$ reduces to the ordinary Euclidean modulus, i.e. $|q|=\sqrt{q_1q_1}$. A quaternion valued function $f:\mathbb{R}^{2}\rightarrow\mathbb{H}$ can also be written  as
 $$f (x, y): =f_1 (x, y) +if_ {2} (x, y) +jf_3 (x, y)+kf_4 (x, y), $$
 where $(x,y)\in\mathbb{R}\times\mathbb{R}$.\\

             The  inner product of two quaternion valued functions $f, g$ defined on $\mathbb{R}^2$ is defined  as follows $$<f,g>_{L^{2}(\mathbb{R}^2,\mathbb{H})}=\int_{\mathbb{R}^2}f(x)\overline{g(x)}dx$$
When $f=g$, we obtain the associated norm giving by $$\|f\|^{2}_{2}=<f,f>_{2}=\int_{\mathbb{R}^2}|f(x)|^2dx$$
we define the space $L^{2}(\mathbb{R}^2,\mathbb{H})$ of all squared integrable functions by
$$L^2(\mathbb{R}^2,\mathbb{H})=\{ f:\mathbb{R}^2\rightarrow\mathbb{H}| \|f\|_{2}<\infty\}$$
For quaternion measurable function $f$ and $p$ nonzero integer, we define
$$\|f\|_{L^p(\mathbb{R}^2,\mathbb{H})}=\int_{\mathbb{R}^2}|f(x)|^pdx$$

  $$L^p(\mathbb{R}^2,\mathbb{H})=\{ f:\mathbb{R}^2\rightarrow\mathbb{H}| \|f\|_{L^p(\mathbb{R}^2,\mathbb{H})}<\infty\}$$
Convolution of two functions two measurable functions $f$ and $g$ over $\mathbb{R}^2$ is given by,
$$f\ast g(y)=\int_{\mathbb{R}^2}f(x)g(x-y)dx$$

\section{The two-sided  Gabor Quaternionic  Fourier transform (GQFT)}

In this section, we start by defining the two-sided Gabor quaternion Fourier transform GQFT, and we reminder some properties, which will be used to prove the principle results.\\

\begin{definition}[Quaternion Fourier transform]
  The two-sided quaternion Fourier transform (QFT)of a quaternion function
  $f\in L^{1}(\mathbb{R}^{2},\mathbb{H})$ is the function  $\mathcal{F}_{q}(f): \mathbb{R}^{2}\rightarrow \mathbb{H}$  defined by:\\ for $\omega=(\omega_1,\omega_2)\in \mathbb{R}\times\mathbb{R}$
  \begin{equation}\label{QFT}
  \mathcal{F}_{q}(f)(w)=\int_{\mathbb{R}^{2}}e^{-2\pi ix_1.\omega_1}f(x)e^{-2\pi jx_2.\omega_2}dx
  \end{equation}
  where $dx=dx_1dx_2$
\end{definition}
According to \cite{Elkassimi} the GQFT is given by

\begin{definition}\label{def-GQFT}
  We define the GQFT of $f\in L^{2}(\mathbb{R}^2,\mathbb{H})$ with respect to non-zero quaternion window function
  $\varphi\in L^{2}(\mathbb{R}^2,\mathbb{H})$ as,
  \begin{equation} \label{eqGQFT}
 \mathcal{G}_{\varphi}f(\omega,b)=\int_{\mathbb{R}^2}e^{-2\pi ix_1\omega_1}f(x)\overline{\varphi(x-b)}e^{-2\pi jx_2.\omega_2}dx
  \end{equation}
\end{definition}
We note by  $$\|\mathcal{G}_{\varphi}\{f\}\|^2_{L^2(\mathbb{R}^2\times\mathbb{R}^2,\mathbb{H})}=\int_{\mathbb{R}^2}\int_{\mathbb{R}^2}\mathcal{G}_{\varphi}f(\omega,b)d\omega db$$
Some important properties of GQFT have been derived \cite{Elkassimi}, which we will use to prove some uncertainty principle and some inequalities.
\begin{theorem}[Inversion formula]
  Let $\varphi$ be a quaternion window function. Then for every function $f\in L^{2}(\mathbb{R}^2,\mathbb{H})$ can be reconstructed  by :
  $$f(x)=\frac{1}{\|\varphi\|^2_{2}}\int_{\mathbb{R}^2}\int_{\mathbb{R}^2}e^{2i\pi x_1\omega_1}G_\varphi f(\omega,b)e^{2j\pi x_2\omega_2}\varphi(x-b)d\omega db$$
\end{theorem}

\begin{theorem}[Plancherel theorem ]\label{Parseval-GQFT}
  Let $\varphi$ be quaternion window function and\\ $f\in L^{2}(\mathbb{R}^2,\mathbb{H})$, then we have
  \begin{equation}\label{parseval1}
  \|\mathcal{G}_{\varphi}\{f\}\|^2_{L^2(\mathbb{R}^2\times\mathbb{R}^2,\mathbb{H})}=\|f\|^2_{2}\|\varphi\|^2_{2}
  \end{equation}

\end{theorem}

\section{Uncertainty Principles}

\begin{theorem}[Hausdorff-Young inequality]
If $1\leq p\leq 2 $ and letting $p^{'}$ be such that $\frac{1}{p}+\frac{1}{p^{'}}=1$, then, for all $f\in L^{p}$ we have
\begin{equation}\label{hausdorff-QFT}
  \|\mathcal{F}_{q}\{f\}\|_{q,p^{'}}\leq \|f\|_{p}
\end{equation}
  where,
  $$\|\mathcal{F}_{q}\{f\}\|_{q,p^{'}}=\left( \int_{\mathbb{R}^2}|\mathcal{F}_{q}\{f\}(\omega)|_{q}^{p^{'}}d\omega\right)^{\frac{1}{p^{'}}}$$
with
$$|\mathcal{F}_{q}\{f\}(\omega)|_{q}=|\mathcal{F}_{q}\{f_0\}(\omega)|+|\mathcal{F}_{q}\{f_1\}(\omega)|+|\mathcal{F}_{q}\{f_2\}(\omega)|+|\mathcal{F}_{q}\{f_3\}(\omega)|$$
\end{theorem}

\section{Local uncertainty principle}
Now, we  give some versions of local uncertainty principles like Benedicks uncertainty principle. \\
First, we start by giving the following Hausdorff-Young's lemma for the GQFT.

\begin{lemma}
  Let $\varphi\in L^p(\mathbb{R}^2,\mathbb{H})$, $f\in L^q(\mathbb{R}^2,\mathbb{H})$ and $p,q \in[1,+\infty[$ with $\frac{1}{p}+\frac{1}{q}=1$, we have
\begin{equation}\label{young-inequality}
\|\mathcal{G}_{\varphi}\{f\}(\omega,b)\|_{\infty} \leq \|f\|_{L^q(\mathbb{R}^2,\mathbb{H})}\|\varphi\|_{L^p(\mathbb{R}^2,\mathbb{H})}
\end{equation}
\end{lemma}
 \begin{proof}
   We have
\begin{eqnarray*}
  |\mathcal{G}_{\varphi}\{f\}(\omega,b)| &=& |\int_{\mathbb{R}^2}e^{-2\pi ix_1\omega_1}f(x)\overline{\varphi(x-b)}e^{-2\pi jx_2\omega_2}dx| \\
      &\leq& \int_{\mathbb{R}^2} |f(x)||\overline{\varphi(x-y)}|dx
\end{eqnarray*}
Using H\"{o}lder inequality we get our result
\begin{equation*}
\|\mathcal{G}_{\varphi}\{f\}(\omega,b)\|_{L^{\infty}(\mathbb{R}^2 \times\mathbb{R}^2,\mathbb{H})} \leq\|f\|_{L^q(\mathbb{R}^2,\mathbb{H})}\|\varphi\|_{L^p(\mathbb{R}^2,\mathbb{H})}
\end{equation*}
\end{proof}

\begin{theorem}
  Let $\varphi$ a quaternion  windowed function, with $\|\varphi\|_2=1$. For $ f\in L^2(\mathbb{R}^2,\mathbb{H})$ that $\|f\|_2=1$, then for $\Sigma \subset \mathbb{R}^2 \times\mathbb{R}^2$ and $0\geq\varepsilon<1$ such that,
 $$ \int\int_{\Sigma}|\mathcal{G}_{\varphi}\{f\}(\omega,b)|^2d\omega db\geq1-\varepsilon.$$
 we have $m(\Sigma)\geq1-\varepsilon$.\\
 Where $m(\Sigma)$ is the Lebesgue measure of $\Sigma$.
\end{theorem}

\begin{proof}
  For $f\in {L^2(\mathbb{R}^2,\mathbb{H})}$
  we have by the lemma \ref{young-inequality}

  \begin{equation*}
\|\mathcal{G}_{\varphi}\{f\}(\omega,b)\|_{L^{\infty}(\mathbb{R}^2 \times\mathbb{R}^2,\mathbb{H})} \leq\|f\|_{2}\|\varphi\|_{2}
\end{equation*}
 From the relation above, we get
 \begin{eqnarray*}
  1-\varepsilon\leq\int\int_{\Sigma}|\mathcal{G}_{\varphi}\{f\}(\omega,b)|^2d\omega db &\leq&\|\mathcal{G}_{\varphi}\{f\}\|^2_{L^{\infty}(\mathbb{R}^2 \times\mathbb{R}^2,\mathbb{H})} m(\Sigma)\\
   &\leq& m(\Sigma)\|f\|^2_{2}\|\varphi\|^2_{2}\\
   &=&m(\Sigma)
  \end{eqnarray*}
  then, $$1-\varepsilon\leq m(\Sigma)$$

\end{proof}
\begin{theorem}\label{concentration1-theo}
  Let $\Sigma\subset \mathbb{R}^2\times\mathbb{R}^2$ such that $0<m(\Sigma)<1$. Then for all $f, \varphi\in L^2(\mathbb{R}^2,\mathbb{H})$, we have,
  \begin{equation}\label{concentration1}
\|f\|_{L^2(\mathbb{R}^2,\mathbb{H})}\|\varphi\|_{L^2(\mathbb{R}^2,\mathbb{H})} \leq   \frac{1}{\sqrt{1-m(\Sigma)}}\left(\int\int_{\Sigma^{c}}|\mathcal{G}_{\varphi}\{f\}(\omega,b)|^2 d\omega db\right)^{\frac{1}{2}}
  \end{equation}

\end{theorem}

\begin{proof}
  For every function $f\in L^2(\mathbb{R}^2,\mathbb{H})$,
  \begin{eqnarray*}
    \|\mathcal{G}_{\varphi}\{f\}\|^2_{L^2(\mathbb{R}^2\times\mathbb{R}^2,\mathbb{H})} &=& \int_{\mathbb{R}^2}\int_{\mathbb{R}^2}|\mathcal{G}_{\varphi}\{f\}(\omega,b)|^2 d\omega db \\
     &=& \int\int_{\Sigma}|\mathcal{G}_{\varphi}\{f\}(\omega,b)|^2 d\omega db+\int\int_{\Sigma^{c}}|\mathcal{G}_{\varphi}\{f\}(\omega,b)|^2 d\omega db \\
     &\leq& m(\Sigma)\|\varphi\|^2_{2}\|f\|^2_{2}+ \int\int_{\Sigma^{c}}|\mathcal{G}_{\varphi}\{f\}(\omega,b)|^2 d\omega db
  \end{eqnarray*}
 Then,  $$\int\int_{\Sigma^{c}}|\mathcal{G}_{\varphi}\{f\}(\omega,b)|^2 d\omega db\geq \|\mathcal{G}_{\varphi}\{f\}\|^2_{L^2(\mathbb{R}^2\times\mathbb{R}^2,\mathbb{H})}-m(\Sigma)\|\varphi\|^2_{2}\|f\|^2_{2}$$

Using the Plancherel formula \eqref{Parseval-GQFT}, we get

$$\|f\|_{L^2(\mathbb{R}^2,\mathbb{H})}\|\varphi\|_{L^2(\mathbb{R}^2,\mathbb{H})} \leq   \frac{1}{\sqrt{1-m(\Sigma)}}\left(\int\int_{\Sigma^{c}}|\mathcal{G}_{\varphi}\{f\}(\omega,b)|^2 d\omega db\right)^{\frac{1}{2}}$$

\end{proof}
\begin{remark}
  This shows that for a non zero function $f$, if its  Gabor transform  $\mathcal{G}^{\varphi}\{f\}$ is  concentrated on a set $\Sigma$ of volume such that, $0<m(\Sigma)<1 $ then $f\equiv0$ or $\varphi\equiv0$.
\end{remark}

\begin{theorem}
  Let $s>0$. There exists a constant $C_{s}>0$ such that, for $f,\varphi\in L^{2}(\mathbb{R}^2,\mathbb{H})$
  \begin{equation}\label{concentration2}
   \|f\|_{L^{2}(\mathbb{R}^2,\mathbb{H})}\|\varphi\|_{L^{2}(\mathbb{R}^2,\mathbb{H})} \leq C_{s}\left(\int\int_{\mathbb{R}^2\times\mathbb{R}^2}
    {|(\omega,y)|}^{2s}|\mathcal{G}^{\varphi}\{f\}(\omega,y)|^2d\omega dy\right)^{\frac{1}{2}}
  \end{equation}
\end{theorem}

\begin{proof}
  Let $0<r\leq1$ be a real number and $B_{r}=\{(\omega,b)\in \mathbb{R}^2\times\mathbb{R}^2: |(\omega,b)|<r\}$ the ball of center 0 and radius $r$ in $\mathbb{R}^2\times\mathbb{R}^2$. Fix $0<t\leq1$ small enough such that $m(B_{t})<1$. Therefore by the inequality \eqref{concentration1} we obtain
  \begin{eqnarray*}
    \|f\|^2_{L^2(\mathbb{R}^2,\mathbb{H})}\|\varphi\|^2_{L^2(\mathbb{R}^2,\mathbb{H})}
     &\leq&  \frac{1}{t^{2s}(1-m(B_{t}))}\int\int_{|(\omega,b)|>t} t^{2s}|\mathcal{G}_{\varphi}\{f\}(\omega,b)|^2d\omega db   \\
     &\leq&  \frac{1}{t^{2s}(1-m(B_{t}))}\int\int_{|(\omega,b)|>t} {|(\omega,b)|}^{2s}|\mathcal{G}_{\varphi}\{f\}(\omega,b)|^2d\omega db \\
     &\leq&\frac{1}{t^{2s}(1-m(B_{t}))}\int\int_{\mathbb{R}^2\times\mathbb{R}^2}{|(\omega,b)|}^{2s}|\mathcal{G}_{\varphi}\{f\}(\omega,b)|^2
    d\omega db
  \end{eqnarray*}
  we take the square of the two  sided of the inequality
 $$\|f\|^2_{L^2(\mathbb{R}^2,\mathbb{H})}\|\varphi\|^2_{L^2(\mathbb{R}^2,\mathbb{H})}\leq\frac{1}{t^{2s}(1-m(B_{t}))}\int\int_{\mathbb{R}^2\times\mathbb{R}^2}{|(\omega,b)|}^{2s}|\mathcal{G}_{\varphi}\{f\}(\omega,b)|^2
    d\omega db$$
    we get
    $$\|f\|_{L^2(\mathbb{R}^2,\mathbb{H})}\|\varphi\|_{L^2(\mathbb{R}^2,\mathbb{H})}\leq\frac{1}{t^{s}\sqrt{1-m(B_{t})}}\left(\int\int_{\mathbb{R}^2\times\mathbb{R}^2}{|(\omega,b)|}^{2s}|\mathcal{G}_{\varphi}\{f\}(\omega,b)|^2
    d\omega db\right)^{\frac{1}{2}}$$
    We obtain the desired result by taking $C_{s}=t^{s}\sqrt{1-m(B_{t})}$
\end{proof}

\subsection{Benedicks-type uncertainty principle}\leavevmode\par
 First, we remind the definitions of some notions.
 \begin{definition}
   Let $\Sigma$ a measurable subset of $\mathbb{R}^2 \times\mathbb{R}^2$ and $\varphi\in L^2(\mathbb{R}^2,\mathbb{H})$ a nonzero window function. Then, \\
  $(1)$ We say that $\Sigma$ is weakly annihilating, if any function $f\in L^2(\mathbb{R}^2,\mathbb{H})$ vanishes when its GQFT $\mathcal{G}_{\varphi}\{f\}$ with respect to window $\varphi$ is supported in $\Sigma$. \\
  $(2)$ We say that $\Sigma$ is strongly annihilating, if there exists a  constant $C(\Sigma)>0$, such that for every function $f\in \mathcal{G}_{\varphi}\{f\}$
   $$\|f\|^2_{L^2(\mathbb{R}^2,\mathbb{H})}\|\varphi\|^2_{L^2(\mathbb{R}^2,\mathbb{H})}\leq C(\Sigma)\int\int_{\mathbb{R}^2\times\mathbb{R}^2}|\mathcal{G}_{\varphi}\{f\}(\omega,b)|^2 d\omega db$$
   The constant $C(\Sigma)$ will be called the annihilation constant of $C(\Sigma)$
 \end{definition}
\begin{lemma}
  Let $\varphi$ be a nonzero window function. Then, \\
  $(1)$ If $\|P_{\Sigma}P_{\varphi}\|<1$ , then for all $f\in L^2(\mathbb{R}^2,\mathbb{H})$,
  \begin{equation}\label{concentration 2}
     \|f\|^2_{L^2(\mathbb{R}^2,\mathbb{H})}\|\varphi\|^2_{L^2(\mathbb{R}^2,\mathbb{H})}\leq \frac{1}{\sqrt{1-\|P_{\Sigma}P_{\varphi}\|^2}}\int\int_{\Sigma^{c}}|\mathcal{G}_{\varphi}\{f\}(\omega,b)|^2 d\omega db
  \end{equation}
  $(2)$ If $\Sigma$ is strongly annihilating, $\|P_{\Sigma}P_{\varphi}\|<1$
\end{lemma}
\begin{proof}
For every $f\in L^{2}(\mathbb{R}^2,\mathbb{H})$; we have
  \begin{equation}\label{concentration12}
   \|G_{\varphi}(f)\|^{2}_{L^2(\mathbb{R}^2,\mathbb{H})}=\|\mathcal{G}_{\varphi}\{f\}\chi_{\Sigma}\|^{2}_{L^2(\mathbb{R}^2\times\mathbb{R}^2,\mathbb{H})}+\|\mathcal{G}_{\varphi}\{f\}\chi_{\Sigma^{c}}\|^{2}_{L^2(\mathbb{R}^2\times\mathbb{R}^2,\mathbb{H})}
  \end{equation}

  with $\mathcal{G}_{\varphi}\{f\}\chi_{\Sigma}=P_{\Sigma}P_{\varphi}(\mathcal{G}_{\varphi}\{f\})$\\
  and from the Plancherel formula \ref{Parseval-GQFT}
  \begin{eqnarray*}
    \|\mathcal{G}_{\varphi}\{f\}\chi_{\Sigma}\|^{2}_{L^2(\mathbb{R}^2\times\mathbb{R}^2,\mathbb{H})} &\leq& \|P_{\Sigma}P_{\varphi}\|^2 \|\mathcal{G}_{\varphi}\{f\}\|^2_{L^2(\mathbb{R}^2\times\mathbb{R}^2,\mathbb{H})}\\
     &=& \|P_{\Sigma}P_{\varphi}\|^2 \|\varphi\|^2_{L^{2}(\mathbb{R}^2,\mathbb{H})} \|f\|^2_{L^{2}(\mathbb{R}^2,\mathbb{H}))}\\
     &=& \|\varphi\|^2_{L^{2}(\mathbb{R}^2,\mathbb{H})}\|P_{\Sigma}P_{\varphi}\|^2  \|f\|^2_{L^{2}(\mathbb{R}^2,\mathbb{H})},
  \end{eqnarray*}

 Thus, by the equation \eqref{concentration12}
 $$ \|\mathcal{G}_{\varphi}\{f\}\chi_{\Sigma}\|^{2}_{L^2(\mathbb{R}^2\times\mathbb{R}^2,\mathbb{H})}\geq(1-\|P_{\Sigma}P_{\varphi}\|^2) \|\varphi\|^2_{L^{2}(\mathbb{R}^2,\mathbb{H})} \|f\|^2_{L^{2}(\mathbb{R}^2,\mathbb{H})}$$
 $$\displaystyle\|f\|_{L^2(\mathbb{R}^2,\mathbb{H})} \|\varphi\|_{L^{2}(\mathbb{R}^2,\mathbb{H})}\leq \frac{1}{\sqrt{1-\|P_{\Sigma}P_{\varphi}\|^2}}\|\mathcal{G}_{\varphi}\{f\}\chi_{\Sigma}\|_{L^2 (\mathbb{R}^2\times\mathbb{R}^2,\mathbb{H})}$$

\end{proof}
Now, we give an analogue of Benedicks-type theorem,  which state that, for a subset $\Sigma$ of the form $\Sigma=S\times B_{R}\subset \mathbb{R}^2 \times\mathbb{R}^2$, such that $0<m(S)<\infty$ and $B_{R}$ is the ball of centre 0 and the radius $R$.\\
 Then $\Sigma$ is weakly annihilating.\\
 Now, we give the Benedicks theorem,
 \begin{theorem}[Benedicks-type theorem for $\mathcal{G}_{\varphi}\{f\}$ ]
   Let $r,R>0$. Let $\varphi\in L^2(\mathbb{R}^2,\mathbb{H})\cap L^{\infty}(\mathbb{R}^2,\mathbb{H})$ be nonzero window function such that $supp{\varphi}\subseteq B_{r}$
   and let $\Sigma=S\times B_{R}\subset \mathbb{R}^2 \times\mathbb{R}^2$, be a subset of finite measure $0<m(\Sigma)<\infty$. Then,
   $$ImP_{\varphi}\cap ImP_{\Sigma}=\{0\}$$
   i.e, $\Sigma$ is weakly annihilating.
 \end{theorem}
\begin{proof}
  Let $F\in ImP_{\varphi}\cap ImP_{\Sigma}$, then, there exists a function $f\in L^2(\mathbb{R}^2,\mathbb{H})$, such that, $F=\mathcal{G}_{\varphi}\{f\}$ and $supp \{F\}\subset \Sigma$.\\
  Then, $$F(\omega,b)=\mathcal{F}_{q}(f(.)\overline{\varphi(.-b)})(\omega)$$
  thus, $$supp\{f(.)\overline{\varphi(.-b)}\}\subset S.$$
   On other hand $supp\{\varphi\}\subset B_{r}$,\\
  we have $$supp\{f(.)\overline{\varphi(.-b)}\}\subset B_{r+R},$$
  Hence, by the  Benedicks theorem for two sided quaternion Fourier transform  \cite{Benedicks},\\
  we deduce that $$f(.)\overline{\varphi(.-b)}\equiv 0,\quad then \quad F=0.$$
\end{proof}


\begin{thebibliography}{6}

\bibitem{Assefa}
 D. Assefa, L. Mansinha,  KF. Tiampo, et al., \emph{ Local quaternion Fourier transform and color image
texture analysis}. Signal Process. 90(6) (2010),1825-1835.\\


\bibitem{Bahri1}
 M. Bahri, R. Ashino, \emph{A Variation on Uncertainty Principle and Logarithmic Uncertainty Principle for Continuous Quaternion Wavelet Transforms}, Hindawi, Abstract and Applied Analysis, Volume 2017, Article ID 3795120.\\


\bibitem{Bahri}
M. Bahri, E. Hitzer, R. Ashino, R. Vaillancourt, \emph{ Windowed Fourier transform
of two-dimensional quaternionic signals}. Applied Mathematics and Computation, 216 (2010), 2366-2379.\\


\bibitem{Bahri2}
M. Bahri, Resnawati, and S. Musdalifah, \emph{A Version of Uncertainty Principle for Quaternion Linear Canonical Transform,} Abstract and Applied Analysis, vol. 2018, Article ID 8732457.\\


\bibitem{Bas}
 P. Bas, N. Le Bihan, J.M. Chassery, \emph{Color image watermarking using quaternion Fourier transform}. In: 2003 IEEE International Conference on Acoustics, Speech, and Signal Processing. Proceedings. (ICASSP'03). Vol. 3. Hong Kong: IEEE; (2003).\\

\bibitem{Bayro}
E. Bayro-Corrochano, N. Trujillo, M. Naranjo,  \emph{Quaternion Fourier descriptors for the preprocessing and recognition of spoken words using images of spatiotemporal representations}. J. Math. Imaging Vision., 28(2) (2007), 179-190.\\



\bibitem{Benedicks}
 C. Li-Ping , K. Kit Ian Kou, L. Ming-Sheng,   \emph{Pitt's inequality and the uncertainty principle associated with the quaternion Fourier transform}. J. Math. Anal. Appl. 423 (2015) 681-700.\\

\bibitem{Bulow}
 T. B\"{u}low, \emph{Hypercomplex spectral signal representations for the processing and analysis of images}. Kiel: Universit\"{a}t Kiel. Institut f\"{u}r Informatik und Praktische Mathematik, (1999).\\




\bibitem{Bulow1}
T. B\"{u}low, M. Felsberg, G. Sommer, \emph{Non-commutative hypercomplexe Fourier transforms of multidimensinal signals}, in : G.Sommer(Ed).Geom. Comp. with Cliff Alg., theor. Found. and Appl. in comp. Vision and Robotics, Springer, (2001), 187-207.\\


\bibitem{Chen}
B. Chen, G. Coatrieux , G. Chen , et al., \emph{Full 4-D quaternion discrete Fourier transform based watermarking for color images}. Digital Signal Process. (28)(2014),106-119.\\

\bibitem{Li-ping}
L.P Chen, Kit Ian Kou, Ming-Sheng Liu, \emph{Pitt's inequality and the uncertainty principle associated with the quaternion Fourier transform}
J .Math. Ana. App., (423)(2015),681-700.\\



\bibitem{Donoho}
D.L. Donoho and P.B. Stark, \emph{Uncertainty principles and signal recovery}, SIAM J. Appl. Math., 49 (1989), 906-931.\\



\bibitem{Gabor1}
D. Gabor,  \emph{Theory of communications}, J. Inst. Electr. Eng. London,  (93)(1946), 429-457.\\


\bibitem{Groki}
K.  Gr\"{o}chenig, \emph{Foundations  of  Time-Frequency  Analysis.}(2001) Birkh\"{a}user, Boston, MA.\\




\bibitem{Elkassimi}
M. El kassimi, S. Fahlaoui \emph{The two-sided Gabor quaternion Fourier transform and some uncertainty principles},  Applied and Numerical Harmonic Analysis (In Press),  	arXiv:1901.01098. 	\\

\bibitem{Ell0}
 TA. Ell, S.J. Sangwine, \emph{Hypercomplex Fourier transforms of color images}. IEEE Trans Image Process, 16(1) (2007),22-35.\\



\bibitem{Ell}
 T.A. Ell, \emph{ Quaternion-Fourier transfotms for analysis of two-dimensional linear
time-invariant partial differential systems}. In: Proceeding of the 32nd Confer-ence on Decision and Control, San Antonio, Texas,(1993), 1830-1841. \\



\bibitem{Fu}
Y. Fu, U. K\"{a}hler, P. Cerejeiras, \emph{ The Balian-Low Theorem for the Windowed Quaternionic Fourier Transform }, Adv. Appl. Clifford Algebras, (22)(2012). \\


\bibitem{Hitzar}
E. Hitzer, \emph{Quaternion Fourier transform on quaternion fields and generalisations},  Adv. App. Clifford Algebr. 20 (2010),271-284.\\


\bibitem{Kou}
K. I. Kou and J. Morais, \emph{Asymptotic behaviour of the quaternion linear canonical transform and the Bochner-Minlos theorem,}
Applied Mathematics and Computation,  247(15)(2014), 675-688.\\



\bibitem{Lieb}
E.H. Lieb, \emph{Integral bounds for radar ambiguity functions and Wigner distributions}. Journal of Mathematical Physics, 31(3)(1990), 594-599.\\


\bibitem{Mejj}

H. Mejjaoli and S. Makren,
\emph{Uncertainty principles for the Weinstein transform}, Czechoslovak Mathematical Journal,  61 (4) (2011), 941-974.\\


\bibitem{Ding}
 S. C. Pei, J. J. Ding,  J. H. r.  Chang,   \emph{ Efficient  implementation  of  quaternion
Fourier transform, convolution, and correlation by 2-D complex FFT}, IEEE Trans. Signal Process. 49 (11) (2001) , 2783-2797.\\


\bibitem{Sangwine}
S. J. Sangwine, T. A. Ell, \emph{Hypercomplex Fourier transforms of color images}, IEEE Transactions on Image Processing, 16 (1) (2007), 22-35.\\




  \end{thebibliography}
\end{document}